\documentclass[11pt]{amsart}
\usepackage{setspace}
\usepackage{amssymb, enumitem}
\usepackage[all]{xy}
\usepackage{hyperref, aliascnt, graphicx}

\newtheorem{lma}{Lemma}[section]

\newaliascnt{thmCt}{lma}
\newtheorem{thm}[thmCt]{Theorem}
\aliascntresetthe{thmCt}

\newaliascnt{corCt}{lma}

\aliascntresetthe{corCt}

\newaliascnt{prpCt}{lma}

\aliascntresetthe{prpCt}

\newtheorem*{thm*}{Theorem}
\newtheorem*{cor*}{Corollary}
\newtheorem*{prop*}{Proposition}

\theoremstyle{definition}

\newaliascnt{pgrCt}{lma}
\newtheorem{pgr}[pgrCt]{}
\aliascntresetthe{pgrCt}

\newaliascnt{dfnCt}{lma}

\aliascntresetthe{dfnCt}

\newaliascnt{rmkCt}{lma}

\aliascntresetthe{rmkCt}

\newaliascnt{rmksCt}{lma}

\aliascntresetthe{rmksCt}

\newaliascnt{exaCt}{lma}

\aliascntresetthe{exaCt}

\newaliascnt{egCt}{lma}

\aliascntresetthe{egCt}

\newaliascnt{qstCt}{lma}

\aliascntresetthe{qstCt}

\newaliascnt{pbmCt}{lma}

\aliascntresetthe{pbmCt}

\newaliascnt{ntnCt}{lma}

\aliascntresetthe{ntnCt}

\numberwithin{equation}{section}

\keywords{groupoid, inverse semigroup, representation}
\author{Marat Aukhadiev}
\title{Groupoids viewed as inverse semigroups}

\address{University of M{\"u}nster, Einsteinstr. 61, 48149 M{\"u}nster, Germany.\newline m.aukhadiev@uni-muenster.de}

\begin{document}

%\begin{abstract} 
%\end{abstract}

\maketitle

\section{Groupoids viewed as inverse semigroups}

Apart from the well-known construction of a universal groupoid of an inverse semigroup described in \cite{Paterson}, there exists one more simple connection between groupoids and inverse semigroups.
 We recall the definitions of an inverse semigroup and a groupoid from \cite{Paterson}.

\begin{pgr}Let $P$ be a semigroup. Elements $x$ and $x^*$ in $P$ are called \emph{mutual inverses} if $$xx^*x=x,\mbox{ and } x^*xx^*=x^*.$$ The semigroup $P$ is called an \emph{inverse semigroup} if for any $x\in P$ there exists a \textbf{unique}  inverse element $x^*\in P$. Further, $P$ always stands for an inverse semigroup. 

\begin{thm}\label{vagner}(V.~V.~Vagner \cite{Vagner}). For a semigroup $S$ in which every element has an inverse, uniqueness of inverses is equivalent to the requirement that all idempotents in $S$ commute.
\end{thm}

The set of idempotents of an inverse semigroup $P$ forms a commutative semigroup denoted $E(P)$. In fact, $$E(P)=\{ xx^*| x\in P\}=\{x^*x| x\in P \}.$$ 
\end{pgr}

\begin{pgr} A groupoid is a set $G$ together with a subset $G^2\subset G\times G$, a product map $(a,b)\mapsto ab$ from $G^2$ to $G$, and an inverse map $a\mapsto a^{-1}$ (so that $(a^{-1})^{-1}=a$) from $G$ onto $G$ such that:
\begin{enumerate}
	\item if $(a,b)$, $(b,c)\in G^2$, then $(ab,c)$, $(a,bc)\in G^2$ and $(ab)c=a(bc)$;
	\item $(b,b^{-1})\in G^2$ for all $b\in G$, and if $(a,b)$ belongs to $G^2$, then $a^{-1}(ab)=b$ and $(ab)b^{-1}=a$.
\end{enumerate}
Define source and range maps $d,r$ on $G$ by 
$$d(x)=x^{-1}x,\ r(x)=xx^{-1}.$$
Then for any $x\in G$ we have $d(x^{-1})=r(x)$, $r(x^{-1})=d(x)$, $x=xd(x)=r(x)x$. One can check that $(x,y)\in G^2$ if and only if $d(x)=r(y)$. Denote by $G^0$ the image of $d$ (and $r$). For any $x\in G^0$ we have $x^{-1}=x^2=x$.
\end{pgr}

\begin{lma} For any groupoid $G$ there exists an inverse semigroup $S(G)$, such that as a set $S(G)=G\cup\{0\}$ and multiplication in $S(G)$ extends the multiplication in $G$.
\end{lma}
\begin{proof} As a set take $S(G)=G^1\cup \{0\}$. For any pair $(a,b)\in G\times G \setminus G^2  $ set $a\cdot b=0$. For every $a\in G$ set  $a^*=a^{-1}$, and $0^*=0$. The relation $aa^*a=a$ then follows immediately. Recall, any idempotent in $G$ is of the form $aa^{-1}$ and $(a^{-1},b)\in G^2$ only if $aa^{-1}=d(a^{-1})=r(b)=bb^{-1}$. Therefore, for any $a,b\in S(G)$ either $a^*b=0$ and then $aa^*bb^*=0$, or $aa^*bb^*=aa^*=bb^*=bb^*aa^*$. Therefore, any two idempotents in $S(G)$ are either orthogonal or equal. Hence, $S(G)$ is an inverse semigroup.
\end{proof}

\begin{lma} For any inverse semigroup $S$ with a zero and without a unit, in which all projections are mutually orthogonal, there exists a groupoid $G(S)=S\setminus\{0\} $ with a multiplication restricted from $S$. 
\end{lma}
\begin{proof} Indeed, define the groupoid $G=G(S)$ by formulas
$$G=S\setminus\{0\},\ G^{(2)}=\{(x,y)\in S\times S| \ x\cdot y\neq 0\},$$
$$x\cdot y=x{\cdot}_{S} y, \  x^{-1}=x^*.$$ 
Then, as usual, we get range and source maps: $r(x)=xx^*$, $d(x)=x^*x$ and the set of units $G^{(0)}=\{xx^*|\ x\in S,\ x\neq 0\}$, the set of non-zero idempotents of $S$. For any $x,y\in S$ we have $xy=0$ if and only if $x^*xyy^*=0$. Therefore, $(x,y)\in G^{(2)}$ if and only if $x^*x=yy^*$. Hence, for such $x,y$ one has $x^*xy=y$, $xyy^*=x$. Now the last thing is to check associativity. Take $x,y,z\in S$, such that $(x,y), (y,z)\in G^{(2)}$ and suppose $(xy)z=xyz=0$. Then one has $$x^*xyz=0 \Rightarrow yy^*yz=0 \Rightarrow yz=0,$$ which is a contradiction.
\end{proof}

We see that algebraically groupoids form a special class of inverse semigroups: inverse semigroups with zero and without a unit, and where all projections are mutually orthogonal. It is easy to verify that if $G$ is a discrete groupoid,  then any representation of $G$ generates a *-representation of $S(G)$. The reverse is also true if we require that a *-representation of an inverse semigroup satisfies $\pi(0)=0$. Hence, the C*-algebra $C^*(G)$ is isomorphic to $C^*(S(G))$.

Unfortunately, this does not hold for locally compact groupoids or even r-discrete groupoids. The reason is that the multiplication extended from a groupoid $G$ to $S(G)$ is not continuous, so the topology of $G$ does not make $S(G)$ into a topological semigroup. From this point of view, the theory of topological groupoids is a theory of a special class of inverse semigroups with a ``partial'' topology, i.e. a topology given on its subspaces.

\end{document}